\documentclass[12pt]{amsart}
\usepackage{mathrsfs}
\usepackage{graphicx} 

\newcommand{\hide}[1]{}

\usepackage{amssymb}
\usepackage{amsmath}

\usepackage{geometry}
\usepackage{caption}

\usepackage{wrapfig}

\usepackage{float}

\usepackage{bm}

\DeclareMathAlphabet{\mathdutchcal}{U}{dutchcal}{m}{n}
\SetMathAlphabet{\mathdutchcal}{bold}{U}{dutchcal}{b}{n}
\DeclareMathAlphabet{\mathdutchbcal}{U}{dutchcal}{b}{n}

\usepackage{tikz}
\usetikzlibrary{arrows, arrows.meta}

\def\textcolor#1{}

\newcommand{\R}{\mathbb{R}}

\newcommand{\Hh}{\mathbb{H}}

\newcommand{\Ss}{\mathcal{S}}
\newcommand{\Bb}{\mathcal{B}}

\renewcommand{\tilde}{\widetilde}
\renewcommand{\rho}{\varrho}
\renewcommand{\phi}{\varphi}
\newcommand{\eps}{\varepsilon}
\renewcommand{\theta}{\vartheta}

\DeclareMathOperator{\TC}{\mathrm{TC}}

\renewcommand{\ge}{\geqslant}
\renewcommand{\le}{\leqslant}

\usepackage[colorlinks=true,urlcolor=blue]{hyperref}

\theoremstyle{theorem}
\newtheorem{theorem}{Theorem}[section]

\newtheorem*{main}{Main Theorem}

\newtheorem*{BorisenkoConj}{Borisenko's Conjecture}

\newtheoremstyle{Intro}
{}
{}
{\itshape}
{}
{\scshape}
{.}
{.5em}
{}

\theoremstyle{Intro}

\theoremstyle{definition}
\newtheorem{definition}[theorem]{Definition}

\theoremstyle{remark}
\newtheorem*{remark}{\textsc{Remark}}

\newcounter{reminder}

\newtheoremstyle{claim}
  {}
  {}
  {\itshape}
  {0pt}
  {\scshape}
  {.}
  { }
  {\thmname{#1}\thmnumber{ #2}\thmnote{ (#3)}}

\theoremstyle{claim}

\numberwithin{equation}{section}

\title[A reverse isoperimetric inequality in space forms]{A reverse isoperimetric inequality in three-dimensional space forms}

\author[Kostiantyn Drach]{Kostiantyn Drach}
\author[Gil Solanes]{Gil Solanes}
\author[Kateryna Tatarko]{Kateryna Tatarko}
\thanks{\ The first and second authors are partially supported by Agencia Estatal de Investigaci\'on through Grants PID2023-147252NB-I00 and PID2024-157757NB-I00
 respectively, and through the Severo Ochoa and Maria de Maeztu Program for Centers and Units of Excellence in R\&D (CEX2020-001084-M). The third author is partially supported by NSERC Discovery Grant number 2022-02961.}

\date{}

\address{Universitat de Barcelona, Gran Via de les Corts Catalanes, 585, 08007 Barcelona, Spain}
\address{Centre de Recerca Matem\`atica, Edifici C, Campus de Bellaterra, 08193 Bellaterra, Barcelona, Spain}
\email{kostiantyn.drach@ub.edu}
\address{Departament de Matem\`atiques, Universitat Aut\`onoma de Barcelona, 08193 Bellaterra, Barcelona, Spain}
\address{Centre de Recerca Matem\`atica, Edifici C, Campus de Bellaterra, 08193 Bellaterra, Barcelona, Spain}
\email{gil.solanes@uab.cat}
\address{Department of Pure Mathematics, University of Waterloo, Waterloo, ON, N2L 3G1, Canada}
\email{ktatarko@uwaterloo.ca}

\begin{document}

\begin{abstract}
A \emph{$\lambda$-convex body} in a three-dimensional space form $M^3(c)$ of constant curvature $c$ is a compact convex set $K$ whose boundary $\partial K$ has normal curvatures bounded below by a constant $\lambda>0$ (in a weak sense).

Within this class, we prove a sharp reverse isoperimetric inequality: among all $\lambda$-convex bodies in $M^3(c)$, with a fixed surface area, the body of minimal volume is the \emph{$\lambda$-convex lens}, i.e., the domain bounded by two totally umbilical caps of curvature $\lambda$. Moreover, this minimizer is unique.

This result confirms \emph{Borisenko’s Conjecture} in the three-dimensional model spaces of constant curvature for $c\neq 0$, and complements recent progress on the conjecture in the Euclidean case $c=0$. As a by-product, our method also yields an alternative proof of the corresponding reverse isoperimetric inequality in two-dimensional hyperbolic space.
\end{abstract}

\maketitle

\section{Introduction}

 Let $n \ge 2$, and let $M^n(c)$ be an $n$-dimensional model space of constant sectional curvature $c \in \R$. A subset $K \subset M^n(c)$ is convex if a geodesic segment between any two points in $K$ is contained in $K$ (in the spherical space we also assume that $K$ lies in an open hemisphere). We say that $K \subset M^n(c)$ is a convex body if it is a compact, convex subset with non-empty interior. 
 
 For a given $\lambda > 0$, a convex body $K \subset M^n(c)$ is called \emph{$\lambda$-convex} if the principal curvatures $k_1,\ldots,k_{n-1}$ with respect to inward pointing normals of the boundary $\partial K$ satisfy
\[
k_i(p) \ge \lambda \qquad i=1,\ldots,n-1,\quad \forall p \in \partial K.
\]
If the boundary of $K$ is not smooth, then the bound above is understood in the weak sense, which we will make precise in Definition~\ref{Def:LambdaConvex}. A simple but important example of a $\lambda$-convex body is the so-called {\em $\lambda$-convex lens}, which is the domain bounded by two totally umbilical hypersurfaces of normal curvature $\lambda$.

The class of $\lambda$-convex bodies has attracted a lot of attention in convex and discrete geometry in the past two decades, e.g.,  \cite{AAF, AACF, Bezdek2008, Bezdek2012, FRS, JMR, SWY}. In particular, it naturally appears in connection with the study of the Kneser--Poulsen conjecture and bodies of constant width. We refer the reader to the surveys \cite{BLN, BLNP} for a comprehensive overview of $\lambda$-convexity and related topics.

In recent years, the \emph{reverse} isoperimetric problems in the class of $\lambda$-convex bodies have received considerable interest. One of the central open problems in this direction is the following conjecture of Borisenko. We will always denote by $|A|$ the $m$-dimensional volume of a compact submanifold with boundary $A^m\subset M^n(c)$.

\begin{BorisenkoConj}
Let $K \subset M^n(c)$ be a $\lambda$-convex body, $n \ge 2$, and let $L \subset M^n(c)$ be a $\lambda$-convex \emph{lens}. If $|\partial K| = |\partial L|$, then $|K| \ge |L|$, with equality if and only if $K$ is a $\lambda$-convex lens.
\end{BorisenkoConj}

In Euclidean space, this conjecture was established for $n=2$ in \cite{BorDr14} (see also the alternative proof in \cite{FKV}) and, very recently, for $n=3$ in \cite{DT}. For dimensions $n \ge 4$, the conjecture remains completely open (see also \cite{Julian} where the authors prove that the solution to this conjecture cannot have a smooth boundary).

In nonflat space forms, progress has thus far been restricted to the two-dimensional case: the conjecture was fully resolved for all two-dimensional spaces of constant curvature and certain natural generalizations in \cite{BorDr15_1, DrLob, BorNA}.

In the present paper, we make a further nontrivial advance by resolving Borisenko’s conjecture affirmatively in three-dimensional space forms of nonzero curvature. Combined with the Euclidean result of \cite{DT}, this completes the verification of the conjecture in all three-dimensional space forms. Our main result is the following.

\begin{main}[Reverse isoperimetric inequality in nonflat space forms]
\label{Thm:Main1}
Let $K \subset M^3(c)$, $c \neq 0$, be a $\lambda$-convex body, and let $L \subset M^3(c)$ be a $\lambda$-convex lens. If $|\partial K| = |\partial L|$, then $|K| \ge |L|$, with equality if and only if $K$ is a $\lambda$-convex lens.
\end{main}

Our method also yields an alternative proof of the corresponding reverse isoperimetric inequality in the two-dimensional hyperbolic space, originally obtained in \cite{DrLob}.

\section{Preliminaries}

In this section, we aim to recall some general background from convex geometry and $\lambda$-convexity. We denote by $M^n(c)$ the complete simply connected $n$-dimensional Riemannian manifold of constant curvature $c$. We call $M^n(c)$ a \emph{space form} (of curvature $c$). When $c=0$, the space form is the standard Euclidean space $\R^n$. When $c > 0$, it is a standard spherical space $\mathbb S^n(c)$, and when $c < 0$, it is a standard hyperbolic space $\Hh^n(c)$.

For a given $\lambda > 0$, let $\Ss_\lambda \subset M^n(c)$ be a complete totally umbilical hypersurface of constant normal curvature equal to $\lambda$. We will call such hypersurfaces \emph{$\lambda$-spheres} and will use the notation $\Bb_\lambda$ for the convex and topologically closed region bounded by $\Ss_\lambda$; we call $\Bb_\lambda$ a \emph{$\lambda$-ball}. Note that, in the hyperbolic space $\Hh^n(-1)$, the $\lambda$-balls (and spheres) are not always geodesic balls (or spheres), as the following classification suggests:

\begin{enumerate}
\item
If $\lambda > 1$, then $\Ss_\lambda$ is a geodesic sphere of radius $\coth^{-1} \lambda$, and $\Bb_\lambda$ is the geodesic ball. 
\item
If $\lambda = 1$, then $\Ss_\lambda$ is a \emph{horosphere} and $\Bb_\lambda$ is a \emph{horoball}. 
\item
\label{It:Horo}
Finally, if $0<\lambda < 1$, then $\Ss_\lambda$ is an \emph{equidistant hypersurface} (or \emph{equidistant} for short), i.e., the set of points at a given distance from a totally geodesic hyperplane. In the literature, equidistants are also called \emph{hyperspheres}. 
\end{enumerate}

In the Poincar\'e unit ball model of $\Hh^n(-1)$, the $\lambda$-spheres $\Ss_\lambda$ are the intersections of the unit ball with Euclidean spheres that: lie inside the unit ball if $\lambda > 1$, touch the unit sphere tangentially if $\lambda = 1$, and intersect the unit sphere at an angle different from $\pi/2$ if $0<
\lambda < 1$.

\subsection{$\lambda$-convex bodies}	
\label{SSec:LConv}
Fix some $\lambda>0$. 

\begin{definition}[$\lambda$-convex body]
\label{Def:LambdaConvex}
A convex body $K \subset M^n(c)$ is \emph{$\lambda$-convex} if for each $p \in \partial K$ there exists a neighborhood $U_p \subset M^n(c)$ and a $\lambda$-sphere $\Ss_\lambda$ that passes through $p$ in such a way that
\[
U_p \cap \partial K \subset \Bb_\lambda,
\]
where $\Bb_\lambda$ is the corresponding $\lambda$-ball.
\end{definition}

In this case, we call $\Ss_\lambda$ a \emph{supporting $\lambda$-sphere} passing through $p \in \partial K$.

The following classical theorem of Blaschke turns the local condition in Definition~\ref{Def:LambdaConvex} into a global one.

\begin{theorem}[Blaschke's Rolling Theorem, \cite{Bla56,DrBla}]
\label{Thm:Bla}
Let $K \subset M^n(c)$ be a $\lambda$-convex body. Then for every $p \in \partial K$ there exists a supporting $\lambda$-sphere $\Ss_{\lambda}(p)$ bounding the $\lambda$-ball $\Bb_\lambda(p)$ such that $K  \subseteq  \Bb_{\lambda}(p)$. \qed
\end{theorem}

Finally, let us point out an equivalent definition of $\lambda$-convexity which is, in some sense, dual to Definition~\ref{Def:LambdaConvex}. The equivalence can be seen using Theorem~\ref{Thm:Bla} (see also \cite[Lemma 7]{BorDr13} for the spaces of constant curvature, and \cite[Corollary 3.4]{BLNP} for the equivalence in the Euclidean space). A convex body $K \subset M^n(c)$ is said to be $\lambda$-convex if and only if any pair of points in $K$ can be connected by the shortest arc of a curve of constant geodesic curvature $\lambda$ and this arc lies entirely in $K$. From this point of view, $\lambda$-convexity is a generalization of the standard notion of convexity, where we say that a body is convex if and only if the line segment that connects any two points lies entirely inside the body.   
It follows that the whole spindle hypersurface, that is, the locus of all arcs connecting the two points, and the domain bounded by this hypersurface must lie inside the body. 

\begin{definition}[$\lambda$-convex polyhedron]
	\label{Def:LambdaConvexPol}
A \emph{$\lambda$-convex polyhedron} in $M^n(c)$ is a $\lambda$-convex body given as the intersection of finitely many $\lambda$-balls $\Bb_\lambda$. Each \emph{facet} of the $\lambda$-convex polyhedron is the intersection of the corresponding $\Bb_\lambda$ with the boundary of the polyhedron. A facet is of dimension $n-1$. Two facets are either disjoint or intersect. In the latter case, their intersection is called a (lower-dimensional) \emph{face}. The faces of dimension $0$ and $1$ are called {\it vertices} and {\it edges} of the polyhedron, respectively.

When $n=3$, we will denote by $\mathcal E$ and $\mathcal V$ the sets of edges and vertices of a $\lambda$-convex polyhedron $K$. For each edge $E \in \mathcal E$, we will denote by $\beta_E$ the angle between the outer normals at $E$ to the faces intersecting along $E$, and by $\ell_E$ the length of $E$. Also, for each vertex $v \in \mathcal V$, we will let $u(K,v)$ be the region in the unit sphere of $T_pM^3(c)$ covered by the unit outer normals to $K$ at $v$. 
\end{definition}

\begin{definition}[$\lambda$-convex lens]	
A $\lambda$-convex polyhedron with exactly two facets is called a \emph{$\lambda$-convex lens}. 
\end{definition}

\section{The Gauss--Bonnet theorem in constant curvature 3-space}

In this section, we prove the following version of the Gauss--Bonnet theorem for $\lambda$-convex polyhedra in $M^3(c)$.

\begin{theorem}[Gauss--Bonnet for $\lambda$-convex polyhedra]
\label{Thm:GB}
Let $K \subset M^3(c)$ be a $\lambda$-convex polyhedron with the vertex set $\mathcal V$ and the edge set $\mathcal E$.  Then
\begin{equation}
    \label{Eq:GB}
    \left(\lambda^2 + c\right) \, |\partial K| + 2\lambda \sum_{E \in \mathcal E} \ell_E \cdot \tan \frac{\beta_E}{2} + \sum_{v \in \mathcal V}|u(K,v)| = 4\pi,
\end{equation}where  $|u(K,v)|$ is the spherical area of $u(K,v)$.
\end{theorem}

\begin{proof}
Put $M=M^3(c)$, and let $\pi\colon SM\to M$ be the sphere bundle of unit tangent vectors of $M$. Given a convex body $K\subset M$, let $N(K)$ be the set of unit outer normal vectors of $K$. It is a bi-Lipschitz submanifold of $SM$. 

A local frame on $SM$ is a collection $e_0,e_1,e_2\colon U\to SM$ defined on some open set $U\subset SM$ such that $e_0(\xi)=\xi$ and $e_0(\xi),e_1(\xi),e_2(\xi)$ are orthonormal $\forall\xi\in U$.
Given such a local frame consider the differential forms $\omega_{1,0},\omega_{2,0}\in \Omega^1(SM)$ defined by 
\[
\omega_{1,0}(X)=\langle e_1,\nabla_X e_0\rangle,\qquad \omega_{2,0}(X)=\langle e_2,\nabla_X e_0\rangle,\qquad X\in T SM.
\]
Here $TSM$ denotes the tangent bundle of $SM$, and $\nabla$ is the connection on $TSM$ obtained by pulling back the Levi-Civita connection of $TM$ through $\pi\colon SM\to M$. 
It is easy to check that $\omega_{1,0}\wedge\omega_{2,0}$ does not depend on the choice of $e_1,e_2$, so we may define
\[
\TC(K)=\int_{N(K)} \omega_{1,0}\wedge\omega_{2,0}.
\]
By \cite[Theorem 2.1.13]{AleskerFu}, the functional $\TC$ is continuous with respect to the Hausdorff topology on the set of convex bodies of $M$. When $\partial K$ is smooth, $\TC(K)$ is the integral of the product $k_1k_2$ of principal curvatures. Since $k_1k_2+c$ equals the Riemannian curvature of $\partial K$, the Gauss--Bonnet theorem reads
\begin{equation}
\label{Eq:First}
\TC(K)+c|\partial K|=4\pi.
\end{equation}
By continuity, the previous equality holds for all convex bodies $K$, smooth or not.

Let $K$ be a $\lambda$-convex polyhedron and denote by $\mathcal V,\mathcal E,\mathcal F$ the sets of vertices, edges, and facets of $K$. Then
\[
\TC(K)=\lambda^2|\partial K| +\sum_{E\in\mathcal E} \int_{N(K)\cap \pi^{-1}(E)}\omega_{1,0}\wedge\omega_{2,0} + \sum_{v\in\mathcal V} \int_{N(K)\cap \pi^{-1}(v)}\omega_{1,0}\wedge\omega_{2,0}.
 \]
Each term in the last sum is the area of the region $u(K,v)$ of the unit sphere $\pi^{-1}(v)$ covered by outer normal vectors of $K$ at $v$. Let us focus on the second sum. Let $E\in\mathcal E$ be the intersection of two faces $F_1,F_2\in\mathcal F$, and let $\gamma(s), s\in[0,l]$ be an arc-length parametrization of $E$. Denote by $\vec t(s),\vec n(s),\vec b(s)$ the Frenet frame along $\gamma(s)$ and put
\begin{align*}
 e_0&=-\cos(\alpha)\vec{n}+\sin(\alpha)\vec b\\
 e_1&=\sin(\alpha)\vec n+\cos(\alpha)\vec b\\
 e_2&=\vec t.
\end{align*}
Let $\beta$ be the angle between the inner normals $N_1,N_2$ to $F_1,F_2$. We have $e_0\in N(K)$ for $-\beta/2\le  \alpha \le \beta/2$.

Since
\begin{align*}
&\nabla e_0=e_1 d\alpha-\cos(\alpha)\nabla\vec n +\sin(\alpha)\nabla\vec b\\
&\nabla \vec t=k\, ds\,\vec n\\
&\nabla \vec n=-k\, ds\,\vec t\\
&\nabla \vec b=0,
\end{align*}
we have
\begin{align*}
    \omega_{1,0}&=d\alpha\\
    \omega_{2,0}&=k\cos(\alpha)ds\\
    \omega_{1,0}\wedge \omega_{2,0}&=k\cos(\alpha) d\alpha\wedge ds.
\end{align*}
Integration yields
\[
\int_{N(K)\cap\pi^{-1}(E)} \omega_{1,0}\wedge\omega_{2,0} =\int_0^l\int_{-\beta/2}^{\beta/2} k\cos(\alpha)d\alpha ds=2k\, \ell\,\sin\frac\beta2.
\]
On the other hand, we have 
\[
\lambda=\langle \frac{d\vec t}{ds}, N_1\rangle=k\langle\vec n,N_1\rangle = k\cos\frac\beta2.
\]
Therefore,
\[
\int_{N(K)\cap\pi^{-1}(E)} \omega_{1,0}\wedge\omega_{2,0}=2\lambda \ell \tan\frac\beta2.
\]
Altogether, we get
\[
\TC(K)=\lambda^2|\partial K|+2\lambda \sum_{E\in \mathcal E} \ell_E\tan \frac{\beta_E}2 + \sum_{v\in \mathcal V} \left|u(K,v)\right|.
\]
Combining with \eqref{Eq:First}, we obtain \eqref{Eq:GB}.
\end{proof}

\section{Inner parallel bodies and variation of the surface area}
\label{Sec:Derivative}

For a convex body $K \subset M^n(c)$, let $r(K)$ be its inradius (i.e., the radius of the largest ball contained in $K$). For each $t\geq0$,  the \emph{inner parallel body} $K_t$ \emph{at distance $t$} is defined as 
\[
K_t := \left\{x \in M^n(c) : B(x,t) \subseteq K\right\},
\]
where $B(x,t) \subset M^n(c)$ is a closed geodesic ball centered at $x \in M^n(c)$ of radius $t$.  Note that $K_t$ is empty for $t>r(K)$. This definition coincides with the standard definition of the inner parallel set in Euclidean space. The following relation between the volume and the surface of $K_t$ was obtained for $c=0$ in \cite{matheron}.

\begin{theorem}
\label{Eq:Der}
    If $K_t$ is the inner parallel body at distance $t \ge 0$ of a convex body $K \subset M^n(c)$, $n \ge 2$, then
    \begin{equation}
    \label{Eq:Volume}
    |K| = \int \limits_0^{r(K)} |\partial K_t| \, dt.
\end{equation}
\end{theorem}

\begin{proof}
    Let $U\subset \mathbb R^n$ be a local chart of $M^n(c)$ containing $K$ and such that the Riemannian metric is $ds^2=(h(x))^2(dx_1^2+\cdots+dx_n^2)$ with $h(x)=\frac2{(1+c|x|^2)}$. 
    Consider the function $f\colon K\to \R$ given by $f(x)=d_M(x,\partial K)$, the Riemannian distance from $x$ to $\partial K$. Clearly,
    \[
    d_M(y,\partial K)\leq d_M(y,x)+ d_M(x,\partial K),\qquad d_M(x,\partial K)\leq d_M(x,y)+ d_M(y,\partial K)
    \]
    for any $x,y\in K$, which yields 
    \begin{equation}\label{eq:lipschitz}
     |f(x)-f(y)|=|d_M(x,\partial K)-d_M(y,\partial K)|\leq d_M(x,y).   
    \end{equation}
    For a uniform constant $C$ we have $d_M(x,y)<C|y-x|$ for all $x,y\in K$, and therefore $f$ is Lipschitz. We can thus apply the coarea formula \cite[3.2.12]{federer} to get
    \begin{equation}\label{eq:coarea}
    \int_{\R^n}g(x) Jf(x) dx=\int_0^{r(K)}\int_{\partial K_t} g(x) d\mathcal H^{n-1}(x) dt
    \end{equation}
    where $\mathcal H^{n-1}$ is the Hausdorff measure with respect to the Euclidean metric, $Jf$ is the Euclidean Jacobian of $f$, and $g$ is any integrable function.
    
    By Rademacher's theorem, $f$ is differentiable at almost every $x\in K$. Given such $x$, note that $Jf(x)\leq h(x)$ by \eqref{eq:lipschitz}. Moreover, if $\gamma(t)$ is a geodesic realizing $d_M(x,\partial K)$ (i.e. such that $\gamma(0)=x$  and  $\gamma(f(x))\in\partial K$), then
    \[
    \left. \frac{d}{dt} \right|_{t=0} f(\gamma(t))=-1.
    \]
It follows that $Jf(x)=\frac1{|\gamma'(0)|}=h(x)$. Hence, taking  $g(x)=(h(x))^{n-1}$, equality \eqref{eq:coarea} becomes
\[
|K|=\int_K \left(h(x)\right)^n dx=\int_0^{r(K)}\int_{\partial K_t} \left(h(x)\right)^{n-1} d\mathcal H^{n-1}(x) dt=\int_0^{r(K)} |\partial K_t|dt.
\]    
\end{proof}

 By \cite[Theorem B]{DT}, if $K$ is a $\lambda$-convex body and $L$ is a $\lambda$-convex lens, then 
\begin{equation}
    \label{Eq:Inrad}
    r(K) \ge r(L) \quad \text{ provided }\quad |\partial K| = |\partial L|,
\end{equation}
with equality if and only if $K$ is a $\lambda$-convex lens.  In the 2-dimensional case this result was proved in  \cite{MilInradius} and later extended to Aleksandrov spaces in  \cite{Dr}.

We need to recall the results of \cite[Section 3.1]{DT}, where the variation formula for inner parallel $\lambda$-convex polyhedra is derived in the Euclidean space. For model spaces, the proof is similar; we present only some of the key steps.

By \cite[Corollary 3.3]{gray}, the inner parallel body $K_t$ of a $\lambda$-convex polyhedron $K$ is a $\lambda(t)$-convex polyhedron with 
\[
\lambda'(t)=\lambda(t)^2+c.
\]

We will denote the set of edges of $K_t$ as $\mathcal E(t)$.  As before, given $E \in \mathcal E(t)$ we will denote by $\ell_E$ the length of $E$, and by $\beta_E$ the angle between the outer normals at $E$.

\begin{theorem}[Variation of the surface area for inner parallel polyhedra]
    \label{Thm:Variation}
    If $K \subset M^n(c)$ is a $\lambda$-convex polyhedron, then except for a finite number of $t_0 \in [0,r(K)]$,
    \[
    \left.\frac{d}{dt}\right|_{t=t_0} |\partial K_t| = -(n-1)\lambda(t_0)\cdot |\partial K_{t_0}| - 2\sum_{E \in \mathcal E(t_0)} \ell_E \cdot \tan\frac{\beta_E}{2}.
    \]
\end{theorem}

\begin{proof}
    The proof follows the steps of \cite[Section 3.1]{DT}. Assume that $K_t$ has the same combinatorial structure for all $t$ near $t_0$, which is the case except for finitely many $t_0\in [0,r(K)]$.

    If $F_i(t)$ is the $i$-th smooth facet of $K_t$, then the claim will follow once we show that
    \begin{equation}
    \label{Eq:Area0}
    \left.\frac{d}{dt}\right|_{t=t_0}|F_i(t)| = -(n-1)\lambda(t_0) \cdot |F_i(t_0)| - \sum_{j} \ell_{E_{ij}} \cdot \tan \frac{\beta_{E_{ij}}}{2},
    \end{equation}
where the sum is taken over all edges $E_{ij}$ adjacent to $F_i(t_0)$.

For $t>s > t_0$, let $F_i(t,s)$ be the outer-parallel hypersurface at distance $t-s$ from $F_i(t)$; i.e.,
\[
F_i(t,s) = \{\exp_x((t-s) N(x)) : x \in F_i(t)\},
\]
where $\exp_x$ is the exponential map, and $N(x)$ is the outer normal to $F_i(t)$ at $x$. Put also $\tilde F_i(t)=F_i(t,t_0)$, and note that $\tilde F_i(t)$ is a convex polyhedron in a fixed complete totally umbilical hypersurface $\mathcal B_{i,t_0}$ of curvature $\lambda(t_0)$ containing $F_i(t_0)$. By the first variation of the area formula, we have
\[
\frac{d}{ds}| F_i(t,s)|=(n-1)\lambda(s)|F_i(t,s)|.
\]
Solving for $|F_i(t,s)|$ at $s = t_0$ in the previous ODE  we get 
\[
|\tilde F_i(t)|=|F_i(t,t_0)|=|F_i(t)|\exp\left((n-1)\int_{t_0}^t \lambda(u)du\right). 
\]
It follows that
\begin{equation}\label{Eq:derivative1}
\left.\frac{d}{dt}\right|_{t=t_0}|\tilde F_i(t)| = (n-1) \lambda(t_0) \cdot |F_i(t_0)| + \left.\frac{d}{dt}\right|_{t=t_0}|F_i(t)|.
\end{equation}

On the other hand, for the family $t \mapsto \tilde F_i(t) \subset \mathcal B_{i,t_0}$ we can apply \cite[Theorem 3.3]{DT} to conclude, in the same way as in \cite[Lemma 3.6]{DT}, that
\[
\left.\frac{d}{dt}\right|_{t=t_0}|\tilde F_i(t)| = - \sum_{j} \ell_{E_{ij}} \cdot \tan \frac{\beta_{E_{ij}}}{2}.
\]
Comparing with \eqref{Eq:derivative1} yields \eqref{Eq:Area0}. This completes the proof.\end{proof}

\section{Proof of the Main Theorem}

We are now ready to establish our main result. 

Assume first that $K \subset M^3(c)$, $c \neq 0$, is a $\lambda$-convex polyhedron, and $L$ is a $\lambda$-convex lens such that $|\partial K|=|\partial L|$. Consider the functions $f_K(t) := |\partial K_t|$, $t \in [0,r(K)]$ and $f_L(t) := |\partial L_t|$. Note that $L_t$ is a $\lambda(t)$-convex lens.

We claim that
\begin{equation}
\label{Eq:Key}
f_K(t) \ge f_L(t), \qquad \forall\, t \in [0,r(L)],
\end{equation}
and that this inequality is strict on some nonempty open subinterval of $[0,r(L)]$ unless $K$ is a lens. The latter assertion follows from \eqref{Eq:Inrad}. In view of~\eqref{Eq:Volume}, this claim immediately yields the desired inequality in the Main Theorem, together with the characterization of the equality case. It therefore remains to justify~\eqref{Eq:Key}.

Let $\ell^*_t$ be the length of the unique edge of the lens $L_t$, and $\beta^*_t$ be the angle at this edge between the outer normals of the two faces. Then the Gauss--Bonnet theorem (Theorem~\ref{Thm:GB}) applied to $L_t$ yields
\[
(\lambda(t)^2+c)|\partial L_t|+2\lambda(t) \ell_{t_0}^*\cdot\tan\frac{\beta_{t_0}^*}2=4\pi.
\]
By taking $K=K_t$ in \eqref{Eq:GB}, and comparing with the previous equality, we deduce
\begin{equation*}
    \sum_{E \in \mathcal E(t_0)} \ell_E \cdot \tan \frac{\beta_E}{2} \le 
    \ell^*_{t_0} \cdot \tan \frac{\beta^*_{t_0}}{2},
\end{equation*} 
for every $t_0 \in [0,r(L)]$ such that $|\partial K_{t_0}| = |\partial L_{t_0}|$, and this inequality is strict unless $K$ is a lens. By Theorem~\ref{Thm:Variation} applied to $K_t$ and $L_t$ at $t=t_0$, it then follows that
\begin{equation}
\label{Eq:Ineq}
\begin{aligned}
    f_K'(t_0) = \left.\frac{d}{dt}\right|_{t=t_0} |\partial K_t| &\ge 
    \left.\frac{d}{dt}\right|_{t=t_0} |\partial L_t| = f_L'(t_0),
\end{aligned}
    \end{equation}
    for every $t \in [0,r(L)]$ such that $|\partial K_t| = |\partial L_t|$, and again, this inequality is strict unless $K$ is a lens. In particular, since $|\partial K| = |\partial K_0| = |\partial L| = |\partial L_0|$, and if $K$ is not a lens, we conclude that $f_K(t) > f_L(t)$ in some small right neighborhood $t \in [0,\eps)$. Now assume that \eqref{Eq:Key} is violated. Then there exists a minimal $t_0$ such that $f_K(t) \ge f_L(t), \,\,t \in [0,t_0]$ and  
\[
f_K(t_0) = f_L(t_0), \quad f_K'(t_0) \le f_L'(t_0).
\]
This is a contradiction to \eqref{Eq:Ineq}. Claim~\eqref{Eq:Key} is thus justified, and this finishes the proof of the Main Theorem in the case when $K$ is a $\lambda$-convex polyhedron.

Finally, if $K$ is an arbitrary $\lambda$-convex body, the argument follows \cite[Subsection~3.3]{DT} verbatim by approximation. The construction therein is purely metric, does not depend on the curvature of the ambient space, and relies only on the rigidity case of~\eqref{Eq:Inrad}. \qed

\section{Proof of the reverse isoperimetric inequality in $\Hh^2(-1)$}

In this section, we will demonstrate how our approach gives an alternative proof to the following theorem, originally established in \cite{DrLob}:

\begin{theorem}
    If $K \subset \Hh^2(-1)$ is a $\lambda$-convex body and $L \subset \Hh^2(-1)$ is a $\lambda$-convex lens such that $|\partial K| = |\partial L|$, then $|K| \ge |L|$. Moreover, equality is possible if and only if $K$ is a $\lambda$-convex lens.
\end{theorem}

\begin{proof}
    As in the proof of the Main Theorem, we can focus only on the case when $K$ is a $\lambda$-convex polygon. Let $e_0, \ldots, e_{m-1}$ ($m \ge 2$) be sides of $K$ and $\beta_i$ be the angle between the outer normals at the common vertex of $e_i$ and $e_{i+1}$ (indices are taken modulo $m$).    
    Likewise, let $\beta^*$ be the angle between the outer normals to the sides at the vertices of $L$. 

   We have that 
    \begin{equation}
        \label{Eq:Angles}
        \beta_i \le \beta^*, \quad \forall i.
    \end{equation}
    Indeed, for a fixed $i$, by Blaschke's rolling theorem, $K$ is contained in the set $\tilde{L}$ (compact or not), which is formed as the intersection of two supporting $\lambda$-balls that form sides $e_i$ and $e_{i+1}$. Thus, $|\partial K| \le |\partial \tilde{L}|$ by monotonicity and $|\partial L| \le |\partial \tilde{L}|$. We note that $\beta_i$ is the angle between the outer normals to the sides at a vertex of~$\tilde{L}$. 
    Since the perimeter of $\tilde L$ is non-decreasing with respect to $\beta_i$, we get that $\beta_i \le \beta^*.$

    By the two-dimensional Gauss--Bonnet theorem applied to $K$, we conclude
    \begin{equation}\label{Eq:GB2}
    2\pi = -|K| + \lambda |\partial K| + \sum_{i=0}^{m-1} \beta_i.
    \end{equation}
    Similarly, for the lens $L$, the Gauss--Bonnet formula gives
    \[
    2\pi = -|L| + \lambda |\partial L| + 2 \beta^*.
    \]
    Therefore, we conclude that 
    \begin{equation}
        \label{Eq:Cond}
        \sum_{i=0}^{m-1} \beta_i \le 2\beta^*, \quad \quad \text{whenever }\quad |\partial K| = |\partial L| \quad \text{and} \quad |K| \le |L|. 
    \end{equation}    

As in the proof of the Main Theorem, let $K_t$ be the inner parallel body for $K$ at distance $t \in [0, r(K)]$, and similarly $L_t$. Note that $K_t$ is a $m_t$-sided polygon, with $m_t \le m$. Recall that $r(L) \le r(K)$ provided that $|\partial K| = |\partial L|$ (see \eqref{Eq:Inrad}).

Since for every $t$, both $K_t$ and $L_t$ are $\lambda_t$-convex (for some $\lambda_t$ that can be computed explicitly depending on $t$ and the curvature of the ambient space), we can specify \eqref{Eq:Cond} for every $t_0 \in [0,r(L)]$:
\begin{equation}
        \label{Eq:Cond2}
        \sum_{i=0}^{m_{t_0}-1} \beta_{i,t_0} \le 2\beta^*_{t_0}, \quad \quad \text{whenever }\quad |\partial K_{t_0}| = |\partial L_{t_0}| \quad \text{and} \quad |K_{t_0}| \le |L_{t_0}|. 
    \end{equation}    
Here, $\beta_{i,t}$, $i \in \{0, \ldots, m_t - 1\}$, and $\beta_{*,t}$ are the corresponding angles of $K_t$ and $L_t$.

By Theorem~\ref{Thm:Variation},
\begin{equation}
\label{Eq:Der}
\begin{aligned}
    \left.\frac{d}{dt}\right|_{t=t_0}|\partial K_t| - \left.\frac{d}{dt}\right|_{t=t_0}|\partial L_t| &= 2 \tan \frac{\beta^*_{t_0}}{2} - 2 \sum_{i=0}^{m_{t_0}-1} \tan \frac{\beta_{i,t_0}}{2}
\end{aligned}
\end{equation}
for every $t_0 \in [0, r(L)]$ such that $|\partial K_{t_0}|= |\partial L_{t_0}|$.

Observe that $2x^{-1}\tan(\frac{x}2)$ is a monotone increasing function for $x \in [0, \pi)$. Assuming $|\partial K_{t_0}| = |\partial L_{t_0}|$ and $|K_{t_0}| \le |L_{t_0}|$, we can combine \eqref{Eq:Angles} and \eqref{Eq:Cond2} to conclude that 
\[
\sum_{i=0}^{m_{t_0}-1} \tan \frac{\beta_{i,t_0}}{2} =\sum_{i=0}^{m_{t_0}-1} \frac{\beta_{i,t_0}}{2} \cdot \frac{2}{\beta_{i,t_0}} \tan \frac{\beta_{i,t_0}}{2} \le \frac{2}{\beta^*_{t_0}} \tan \frac{\beta^*_{t_0}}{2} \sum_{i=0}^{m_{t_0}-1} \frac{\beta_{i,t_0}}{2} \le 2\tan \frac{\beta^*_{t_0}}{2}.
\]
Together with \eqref{Eq:Der}, this gives
\begin{equation}
    \label{Eq:Der2}
    \begin{aligned}    
    \left.\frac{d}{dt}\right|_{t=t_0}|\partial K_t| \ge  \left.\frac{d}{dt}\right|_{t=t_0}|\partial L_t| \quad &\text{for every }\quad t_0 \in [0, r(L)] \quad \\
    &\text{ such that } \quad |\partial K_{t_0}|= |\partial L_{t_0}| \quad \text{and} \quad |K_{t_0}| \le |L_{t_0}|. 
    \end{aligned}
\end{equation}
Moreover, the first inequality is strict if the last inequality $|K_{t_0}| \le |L_{t_0}|$ is also strict. 

Towards the contradiction, assume that $|K| <|L|$. Then we can apply \eqref{Eq:Der2} to conclude that 
\[
\left.\frac{d}{dt}\right|_{t=0}|\partial K_t| >  \left.\frac{d}{dt}\right|_{t=0}|\partial L_t|.
\]
Hence, the graph of the function $t \mapsto |\partial L_t|$ lies below the graph of the function $t \mapsto |\partial K_t|$ in some right neighborhood of $0$. If this is true for all $t \in [0,r(L)]$, we conclude that $|K| \ge |L|$, a contradiction. Therefore, there must exist the smallest $t_1 > 0$ such that 
\begin{equation}
\label{Eq:Der3}    
|\partial K_{t_1}| = |\partial L_{t_1}| \quad \text{ and }\quad \left.\frac{d}{dt}\right|_{t=t_1}|\partial K_t| \le   \left.\frac{d}{dt}\right|_{t=t_1}|\partial L_t|
\end{equation}
and the graph of $t \mapsto |\partial L_t|$ goes above the graph of $t \mapsto |\partial K_t|$ in a right neighborhood of $t_1$. By the choice of $t_1$, we know that 
\[
\int_0^{t_1}|\partial K_t|dt \ge \int_0^{t_1}|\partial L_t|dt. 
\]
Thus, using the assumption $|K| < |L|$,
\[
|K_{t_1}| = |K| - \int_0^{t_1}|\partial K_t|dt < |L| - \int_0^{t_1}|\partial L_t|dt = |L_{t_1}|.
\]
Therefore, at $t = t_1$, we have $|K_{t_1}| < |L_{t_1}|$ and $|\partial K_{t_1}| = |\partial L_{t_1}|$. Hence, by \eqref{Eq:Der2}, we must have $\left.\frac{d}{dt}\right|_{t=t_1}|\partial K_t| >  \left.\frac{d}{dt}\right|_{t=t_1}|\partial L_t|$. This is a contradiction to the inequality in \eqref{Eq:Der3}. This is a final contradiction that justifies $|K| \ge |L|$.

The equality case follows as in \cite{DT} (using the rigidity of the inequality in \cite{Dr}).
\end{proof}

\begin{remark}
    Note that the proof above does not work in the $2$-dimensional sphere because in that case the area $|K|$ appears with a different sign in the Gauss--Bonnet formula \eqref{Eq:GB2}.
\end{remark}

\end{document}